\documentclass[12pt]{article}
\usepackage{amsmath,amsthm,amssymb}
\usepackage{amssymb,latexsym}

\newtheorem{theorem}{Theorem}

\newtheorem{lemma}{Lemma}

\begin{document}

\baselineskip=17pt

\title{\bf A quaternary diophantine inequality by prime numbers of a special type}

\author{\bf S. I. Dimitrov}
\date{2017}
\maketitle
\begin{abstract}
Let $1<c<832/825$. For large real numbers $N>0$ and a small constant $\vartheta>0$,
the inequality
\begin{equation*}
|p_1^c+p_2^c+p_3^c+p_4^c-N|<\vartheta
\end{equation*}
has a solution in prime numbers $p_1,\,p_2,\,p_3,\,p_4$ such that, for each $i\in\{1,2,3,4\}$,
$p_i+2$ has at most $32$ prime factors.

\medskip

{\bf  2000 Math.\ Subject Classification}.  11D75, 11N36, 11P32.
\par
\end{abstract}

\section{Introduction and statements of the result.}
\indent

In 1952 I. I. Piatetski-Shapiro \cite{Shapiro} investigated the inequality
\begin{equation}\label{Shapiro}
|p_1^c+p_2^c+\cdot\cdot\cdot+p_r^c-N|<\varepsilon
\end{equation}
where $c>1$ is not an integer, $\varepsilon$ is a fixed small positive number, and
$p_1,...,p_r$ are primes. He proved the existence of an $H(c)$, depending only on
$c$, such that for all sufficiently large real $N$, (\ref{Shapiro}) has a solution for $H(c)\leq r$.
He established that
\begin{equation*}
\limsup\limits_{c\rightarrow\infty}\frac{H(c)}{c\log c}\leq4
\end{equation*}
and also that $H(c)\leq5$ if $1<c<3/2$.

In 1992 Tolev \cite{Tolev1} proved that (\ref{Shapiro}) has a solution for $r=3$ and $1<c<15/14$.
The interval $1<c<15/14$ was subsequently improved by several authors
\cite{Cai}, \cite{Ku-Ne}, \cite{Kumchev}, \cite{Baker-Weingartner}.

In 2003 Zhai and Cao \cite{Zhai-Cao} proved that (\ref{Shapiro}) has a solution for $r=4$ and $1<c<81/68$.
Their result  was improved to $1<c<97/81$ by Mu \cite{Mu}.

In  2016 Dimitrov \cite{Dimitrov} showed that (\ref{Shapiro}) has a solution
for $r=3$, $0<c<4/21$ and primes $p_1,\,p_2,\,p_3$ such that, for each $i\in\{1,2,3\}$,
$p_i+2$ has at most 10 prime factors.

Recently Tolev \cite{Tolev2} proved that (\ref{Shapiro}) has a solution
for $r=3$, $1<c<15/14$ and primes $p_1,\,p_2,\,p_3$ such that, for each $i\in\{1,2,3\}$,
$p_i+2$ has at most $\left[\frac{369}{180-168c}\right]$ prime factors.

Let $P_l$ is a number with at most $l$ prime factors.
Motivated by \cite{Tolev2}, we shall prove the following theorem.
\begin{theorem} Let $A$ be an arbitrary large and fixed, and let $1<c<832/825$.
There exists a number $N_{0}(c)>0$ such
that for each real number $N > N_{0}(c)  $ the inequality
\begin{equation*}
 |p_1^c+p_2^c+p_3^c+p_4^c-N|<\frac{1}{(\log N)^A}
\end{equation*}
has a solution in prime numbers $p_1,\,p_2,\,p_3$  such that
\[
p_1+2=P_{32}^{'}\,,\quad p_2+2=P_{32}^{''}\,,\quad p_3+2=P_{32}^{'''}\,,\quad p_4+2=P_{32}^{''''}\,.
\]
\end{theorem}
By choosing the parameters in a different way we may obtain other similar results, for example
$1<c<51/50$\,,\, $p_i+2=P_r\,,i=1,2,3,4$, where $r$ is large. Obviously the enlargement of the range for $c$ leads to
increase of the number of the prime factors of $p_i+2$.

\section{Notations and some lemmas.}
\indent

As usual $\varphi(n)$ and $\mu(n)$ denote respectively, Euler's function and M\"{o}bius'
function. We denote by $\tau(n)$ the number of the positive divisors of $n$.
Let $(m_1,m_2)$ be the greatest common divisor. Instead of $m\equiv n\,\pmod {k}$
we write for simplicity $m\equiv n\,(k)$. As usual $[y]$ denotes
the integer part of $y$, $e(y)=e^{2\pi \imath y}$.
Let $c$ be a fixed real number such that $1<c<832/825$ and $N$ be a sufficiently large number.
\begin{align}
\label{X}
&X =(N/3)^{1/c}\,;\\
\label{tau}
&\tau=X^{57/275-c}\,;\\
\label{eps}
&\vartheta=\frac{1}{(\log X)^{A+1}}\,,\;\;A>20\hbox{ is arbitrary large}\,;\\
\label{K}
&K=\frac{\log^2X}{\vartheta}\,;\\
\label{D}
&D=X^{1/11-\varepsilon_0}\,,\;\;\varepsilon_0=0.001\,;\\
\label{eta}
&\eta=\frac{\varepsilon_0}{9}\,;\\
\label{z}
&z=X^\beta,\;\;0<\beta<1/33\,;\\
\label{notPz}
&P(z)=\prod\limits_{2<p\le z}p\,,\,\ \mbox{$p$ -prime number}\,;\\
\label{I}
&I(\alpha)=\int\limits_{X/2}^{X}e(\alpha t^c)dt\,.
\end{align}
The value of $\beta$  will be specified latter.

Let $\lambda ^{\pm}(d)$ be the lower and upper bounds Rosser's weights of level
$D$, hence
\begin{equation}\label{Rosser}
|\lambda ^{\pm}(d)|\leq1,\quad\lambda ^{\pm}(d)=0\;\;  \mbox{  if }\;\; d \geq D \;\;\mbox{ or }\;\; \mu(d)=0\,.
\end{equation}
For further properties of Rosser's weights we refer to \cite{Iwa1}, \cite{Iwa2}.
\begin{lemma}\label{Fourier}Let $\vartheta\in \mathbb{R}$ and $k\in \mathbb{N}$.
There exists a function $\theta(y)$ which is $k$ times continuously differentiable and
such that
\begin{align*}
  &\theta(y)=1\quad\quad\quad\mbox{for }\quad\quad|y|\leq 3\vartheta/4\,;\\[6pt]
  &0\leq\theta(y)<1\quad\mbox{for}\quad3\vartheta/4 <|y|< \vartheta\,;\\[6pt]
  &\theta(y)=0\quad\quad\quad\mbox{for}\quad\quad|y|\geq \vartheta\,.
\end{align*}
and its Fourier transform
\begin{equation*}
\Theta(x)=\int\limits_{-\infty}^{\infty}\theta(y)e(-xy)dy
\end{equation*}
satisfies the inequality
\begin{equation*}
|\Theta(x)|\leq\min\bigg(\frac{7\vartheta}{4},\frac{1}{\pi|x|},\frac{1}{\pi|x|}
\bigg(\frac{k}{2\pi |x|\vartheta/8}\bigg)^k\bigg)\,.
\end{equation*}
\end{lemma}
\begin{proof} See \cite{Segal}.
\end{proof}

\begin{lemma}\label{tauest}
Let $n\in \mathbb{N}$. Then
\begin{equation*}
\tau(n)\ll n^\varepsilon\,,
\end{equation*}
where $\varepsilon$ is an arbitrary small positive number.
\end{lemma}
\begin{lemma}\label{log}
Let $X\in \mathbb{R},\,X\geq2$. We have
\begin{equation*}
\sum\limits_{n\leq X}\frac{1}{\varphi(n)}\ll\log X\,.
\end{equation*}
\end{lemma}

\begin{lemma}\label{Iest}
Assume that $F(x)$, $G(x)$ are real functions defined in  $[a,b]$,
$|G(x)|\leq H$ for $a\leq x\leq b$ and $G(x)/F'(x)$ is a monotonous function. Set
\begin{equation*}
I=\int\limits_{a}^{b}G(x)e(F(x))dx\,.
\end{equation*}
If $F'(x)\geq h>0$ for all $x\in[a,b]$ or if $F'(x)\leq-h<0$ for all $x\in[a,b]$ then
\begin{equation*}
|I|\ll H/h\,.
\end{equation*}
If $F''(x)\geq h>0$ for all $x\in[a,b]$ or if $F''(x)\leq-h<0$ for all $x\in[a,b]$ then
\begin{equation*}
|I|\ll H/\sqrt h\,.
\end{equation*}
\end{lemma}
\begin{proof} See (\cite{Titchmarsh}, p. 71).
\end{proof}

\section{Outline of the proof.}
\indent

Consider the sum
\begin{equation}\label{Gamma}
\Gamma= \sum\limits_{X/2<p_1,p_2,p_3,p_4\leq X\atop{ |p_1^c+p_2^c+p_3^c+p_4^c-N|
<\vartheta\atop{(p_{i}+2,P(z))=1,i=1,2,3,4}}}\log p_1\log p_2\log p_3\log p_4\,.
\end{equation}
Any non-trivial estimate from below of $\Gamma$ implies the solvability of
$|p_1^c+p_2^c+p_3^c+p_4^c-N|<\vartheta$ in primes such that $p_i+2=P_h,\;h=[\beta^{-1}]$.

We have
\begin{equation}\label{Gammaest1}
\Gamma\geq\widetilde{\Gamma}=\sum\limits_{X/2<p_1,p_2,p_3,p_4\leq X\atop{(p_i+2,P(z))=1, i=1,2,3,4}}
\theta(p_1^c+p_2^c+p_3^c+p_4^c-N)\log p_1\log p_2\log p_3\log p_4\,.
\end{equation}
On the other hand
\begin{equation}\label{Gammawave}
\widetilde{\Gamma}=\sum\limits_{X/2<p_1,p_2,p_3,p_4\leq X}\theta(p_1^c+p_2^c+p_3^c+p_4^c-N)
\Lambda_1\Lambda_2\Lambda_3\Lambda_4\log p_1 \log p_2\log p_3\log p_4\,,
\end{equation}
where
\begin{equation*}
\Lambda_i=  \sum\limits_{d|(p_{i}+2,P(z))}\mu(d)\;\;,i=1,2,3,4.
\end{equation*}
We denote
\begin{equation}\label{Lambdaipm}
\Lambda_i^{\pm}=  \sum\limits_{d|(p_{i}+2,P(z))}\lambda^{\pm}(d)\;\;,i=1,2,3,4.
\end{equation}
From the linear sieve we know that $\Lambda_{i}^{-}\le \Lambda_{i}\le \Lambda_{i}^{+}$
(see \cite{Bru}, Lemma 10).
Then we have a simple inequality
\begin{equation}\label{Inequality}
\Lambda_1\Lambda_2\Lambda_3\Lambda_4\ge\Lambda_{1}^{-}\Lambda_{2}^{+}\Lambda_{3}^{+}\Lambda_{4}^{+}
+\Lambda_{1}^{+}\Lambda_{2}^{-}\Lambda_{3}^{+}\Lambda_{4}^{+}
+\Lambda_{1}^{+}\Lambda_{2}^{+}\Lambda_{3}^{-}\Lambda_{4}^{+}
+\Lambda_{1}^{+}\Lambda_{2}^{+}\Lambda_{3}^{+}\Lambda_{4}^{-}
-3\Lambda_{1}^{+}\Lambda_{2}^{+}\Lambda_{3}^{+}\Lambda_{4}^{+}\,
\end{equation}
(see \cite{Bru}, Lemma 13).\\
Using (\ref{Gammawave}) and (\ref{Inequality}) we obtain
\begin{align}\label{Gammawaveest}
\widetilde{\Gamma}\geq\Gamma_{0}&=\sum\limits_{X/2<p_1,p_2,p_3,p_4\leq X} \theta(p_1^c+p_2^c+p_3^c+p_4^c-N)\\
&\times(\Lambda_{1}^{-}\Lambda_{2}^{+}\Lambda_{3}^{+}\Lambda_{4}^{+}
+\Lambda_{1}^{+}\Lambda_{2}^{-}\Lambda_{3}^{+}\Lambda_{4}^{+}
+\Lambda_{1}^{+}\Lambda_{2}^{+}\Lambda_{3}^{-}\Lambda_{4}^{+}
+\Lambda_{1}^{+}\Lambda_{2}^{+}\Lambda_{3}^{+}\Lambda_{4}^{-}
-3\Lambda_{1}^{+}\Lambda_{2}^{+}\Lambda_{3}^{+}\Lambda_{4}^{+})\nonumber\\
&\times\log p_1\log p_2\log p_3\log p_4\,.\nonumber
\end{align}
Let
\begin{equation}\label{Gamma0}
\Gamma_{0}=\Gamma_{1}+\Gamma_{2}+\Gamma_{3}+\Gamma_{4}-3\Gamma_{5}\,,
\end{equation}
where for example
\begin{equation}\label{Gamma1}
\Gamma_{1}=\sum\limits_{X/2<p_1,p_2,p_3,p_4\leq X}\theta(p_1^c+p_2^c+p_3^c+p_4^c-N)
\Lambda_1^-\Lambda_2^+\Lambda_3^+\Lambda_4^+\log p_1\log p_2\log p_3\log p_4
\end{equation}
and so on.

It is easy to see that $\Gamma_1=\Gamma_2=\Gamma_3=\Gamma_4$.
We shall consider the sum $\Gamma_1$. The sum $\Gamma_5$ can be considered in the same way.

From \eqref{Lambdaipm} and \eqref{Gamma1} we have
\begin{equation*}
\Gamma_{1}=\sum\limits_{d_{i}|P(z)\atop{i=1,2,3,4}}\lambda^-(d_1)\lambda^+(d_2)\lambda^+(d_3)\lambda^+(d_4)
\sum\limits_{X/2<p_1,p_2,p_3,p_4\leq X\atop{p_{i}+2\equiv0\,(d_i), i=1,2,3,4}}\theta(p_1^c+p_2^c+p_3^c+p_4^c-N)
\end{equation*}
\begin{equation*}
\times\log p_1\log p_2\log p_3\log p_4\,.
\end{equation*}
Using the inverse Fourier transform for the function $\theta(x)$ we get
\begin{align*}
\Gamma_{1}&=\sum\limits_{d_{i}|P(z)\atop{i=1,2,3,4}}\lambda^-(d_1)\lambda^+(d_2)\lambda^+(d_3)\lambda^+(d_4)
\sum\limits_{X/2<p_1,p_2,p_3,p_4\leq X\atop{p_{i}+2\equiv0\,(d_{i}), i=1,2,3,4}}\log p_1\log p_2\log p_3\log p_4\\
&\times\int\limits_{-\infty}^{\infty}\Theta(t)e((p_1^c+p_2^c+p_3^c+p_4^c-N)t)dt\\
&=\int\limits_{-\infty}^{\infty}\Theta(t)e(-N t)L_1(t,X)L^3_2(t,X)dt\,,
\end{align*}
where
\begin{equation}\label{L1}
L_1(t,X)=\sum\limits_{d|P(z)}\lambda^-(d)\sum\limits_{X/2<p\leq X\atop{p+2\equiv0\,(d)}}e(p^{c}t)\log p\,,
\end{equation}
\begin{equation}\label{L2}
L_2(t,X)=\sum\limits_{d|P(z)}\lambda^+(d)\sum\limits_{X/2<p\leq X\atop{p+2\equiv0\,(d)}}e(p^{c}t)\log p\,.
\end{equation}
We divide $\Gamma_{1}$ into  three parts
\begin{equation}\label{Gama1}
\Gamma_1=\Gamma_1^{(1)}+\Gamma_1^{(2)}+\Gamma_1^{(3)}.
\end{equation}
where
\begin{equation}\label{Gamma1,1}
\Gamma_1^{(1)}=\int\limits_{|t|<\tau}\Theta(t)e(-Nt)L_1(t,x)L^3_2(t,X)dt\,,
\end{equation}
\begin{equation}\label{Gamma1,2}
\;\;\;\;\Gamma_1^{(2)}=\int\limits_{\tau\leq|t|\leq K}\Theta(t)e(-N t)L_1(t,X)L^3_2(t,X)dt\,,
\end{equation}
\begin{equation}\label{Gamma1,3}
\Gamma_1^{(3)}=\int\limits_{|t|>K}\Theta(t)e(-N t)L_1(t,X)L^3_2(t,X)dt\,.
\end{equation}
We shall estimate $\Gamma_1^{(3)}$, $\Gamma_1^{(1)}$, $\Gamma_1^{(2)}$ respectively in the sections 4, 5, 6.
In section 7 we shall complete the proof of the Theorem.
\section{Upper bound for $\mathbf{\Gamma_{1}^{(3)}}$.}
\indent

Arguing as in  \cite{Tolev2} we obtain
\begin{lemma} For the sum $\Gamma_{1}^{(3)}$, defined by (\ref{Gamma1,3}), we have
\begin{equation}\label{Gamma1,3est}
\Gamma_{1}^{(3)}\ll1\,.
\end{equation}
\end{lemma}

\section{Asymptotic formula for $\mathbf{\Gamma_{1}^{(1)}}$.}
\indent

The first lemma we need in this section gives us asymptotic formula
for the sums $L_j(\alpha,X)$ denoted by (\ref{L1}) and (\ref{L2}).
\begin{lemma}\label{LAsympt}
Let $D$ is defined by (\ref{D}),
and $\lambda(d)$ be complex numbers defined for $d\le D$ such that
\begin{equation}\label{omega}
|\lambda(d)|\leq1\,,\quad\lambda(d)=0\;\;\mbox{ if }\;\;2|d\;\;\mbox{ or }\;\;\mu(d)=0\,.
\end{equation}
If
\begin{equation*}
L(\alpha,X)=\sum\limits_{d\le D}\lambda(d)
\sum\limits_{\substack{X/2<p\le X \\ p+2\equiv0\,(d)}}e(p^c\alpha)\log p
\end{equation*}
then for $|\alpha|<\tau$ we have
\begin{equation}\label{Lasympt}
L(\alpha,X)=I(\alpha)\sum\limits_{d\leq D}\frac{\lambda(d)}{\varphi(d)}
+\mathcal{O}\bigg(\frac{X}{(\log X)^A}\bigg)\,,
\end{equation}
where $A>0$ is an arbitrary large constant.
\end{lemma}
\begin{proof} See (\cite{Tolev2}, Lemma 10).
\end{proof}

The next lemma is the following
\begin{lemma}\label{intLintI}
Using the definitions (\ref{I}), (\ref{L1}) and (\ref{L2}) we have
\begin{equation*}
\left.\begin{aligned}
&\emph{(i)}\;\;\;\;\;\;\;\;\;\;\;\;\;
\int\limits_{-\tau}^\tau|L_j(\alpha,X)|^2d\alpha\ll X^{2-c}\log^6X\,,\;\; j=1,2\\
&\emph{(ii)}\ \;\;\;\;\;\;\;\;\;\; \int\limits_{-\tau}^\tau|I(\alpha)|^2d\alpha\ll X^{2-c}\log X.
\end{aligned}\right.
\end{equation*}
\end{lemma}
\begin{proof} See (\cite{Tolev2}, Lemma 11).
\end{proof}

Let
\begin{align}
&L_j=L_j(t,X)\,,\,j=1,2\nonumber\\
\label{M1}
&\mathcal{M}_1=\mathcal{M}_1(t,X)
=I(t)\sum\limits_{d\leq D}\frac{\lambda^-(d)}{\varphi(d)}\,,\\
\label{M2}
&\mathcal{M}_2=\mathcal{M}_2(t,X)
=I(t)\sum\limits_{d\leq D}\frac{\lambda^+(d)}{\varphi(d)}\,.
\end{align}
where $L_j(t,X)$ are denoted by (\ref{L1}) and (\ref{L2}).\\
We use the identity
\begin{align}\label{Identity}
L_1L^3_2&=\mathcal{M}_1\mathcal{M}^3_2+(L_1-\mathcal{M}_1)\mathcal{M}^3_2
+L_1(L_2-\mathcal{M}_2)\mathcal{M}^2_2\nonumber\\
&+L_1L_2(L_2-\mathcal{M}_2)\mathcal{M}_2+L_1L^2_2(L_2-\mathcal{M}_2)\,.
\end{align}
Replace
\begin{equation}\label{J1}
J_1=\int\limits_{|t|<\tau}\Theta(t)e(-Nt)\mathcal{M}_1(t,X)\mathcal{M}^3_2(t,X)dt\,.
\end{equation}
Then from Lemma \ref{Fourier}, Lemma \ref{LAsympt}, (\ref{L1}), (\ref{L2}),
(\ref{Gamma1,1}), (\ref{M1}) -- (\ref{J1}) we obtain
\begin{align}\label{Gamma1,1J1}
\Gamma_1^{(1)}-J_1
&=\int\limits_{|t|<\tau}\Theta(t)e(\eta t)\Big(L_1(t,X)-\mathcal{M}_1(t,X)\Big)
\mathcal{M}^3_2(t,X)dt\nonumber\\
&+\int\limits_{|t|<\tau}\Theta(t)e(\eta t)L_1(t,X)
\Big(L_2(t,X)-\mathcal{M}_2(t,X)\Big)\mathcal{M}^2_2(t,X)dt\nonumber\\
&+\int\limits_{|t|<\tau}\Theta(t)e(\eta t)L_1(t,X)L_2(t,X)
\Big(L_2(t,X)-\mathcal{M}_2(t,X)\Big)\mathcal{M}_2(t,X)dt\nonumber\\
&+\int\limits_{|t|<\tau}\Theta(t)e(\eta t)L_1(t,X)L^2_2(t,X)
\Big(L_2(t,X)-\mathcal{M}_2(t,X)\Big)dt\nonumber\\
&\ll\vartheta\frac{X}{(\log X)^A}\Bigg(\int\limits_{|t|<\tau}
|\mathcal{M}^3_2(t,X)|dt+\int\limits_{|t|<\tau}|L_1(t,X)\mathcal{M}^2_2(t,X)|dt\nonumber\\
&+\int\limits_{|t|<\tau}|L_1(t,X)L_2(t,X)\mathcal{M}_2(t,X)|dt
+\int\limits_{|t|<\tau}|L_1(t,X)L^2_2(t,X)|dt\Bigg)\,.
\end{align}
On the other hand (\ref{Rosser}), (\ref{M2}) and Lemma \ref{log} give us
\begin{equation}\label{Mipmest}
|\mathcal{M}_2(t,X)|\ll|I(t)|\log X\,.
\end{equation}
Using (\ref{Gamma1,1J1}) and (\ref{Mipmest}) we find
\begin{align}\label{Gama1,1J1est1}
\Gamma_1^{(1)}-J_1
&\ll\vartheta\frac{X}{(\log X)^{A-3}}\left(\int\limits_{|t|<\tau}|I(t)|^3dt\right.
+\int\limits_{|t|<\tau}|L_1(t,X)||I(t)|^2dt\nonumber\\
&+\int\limits_{|t|<\tau}|L_1(t,X)L_2(t,X)I(t)|dt
\left.+\int\limits_{|t|<\tau}|L_1(t,X)L^2_2(t,X)|dt\right)\,.
\end{align}
Bearing in mind the definitions (\ref{I}), (\ref{L1}) and (\ref{L2}) we get the trivial estimates
\begin{equation}\label{Trivest}
|I(t)|\ll X\,;\quad |L_j(t,X)|\ll X \log^2X\,,\,j=1,2\,.
\end{equation}
Now from (\ref{Gama1,1J1est1}), (\ref{Trivest}) and Lemma \ref{intLintI} we obtain
\begin{equation}\label{Gama1,1J1est2}
\Gamma_1^{(1)}-J_1\ll\vartheta\frac{X^2}{(\log X)^{A-5}}\left(\int\limits_{|t|<\tau}|I(t)|^2dt
+\int\limits_{|t|<\tau}|L_1(t,X)|^2dt\right)\ll\vartheta\frac{X^{4-c}}{(\log X)^{A-11}}\,.
\end{equation}

Let us consider $J_1$. According to Lemma \ref{Iest} we have
\begin{equation}\label{estint}
|I(\alpha)|\ll\frac{X^{1-c}}{|\alpha|}\,.
\end{equation}
Therefore by Lemma \ref{Fourier}, Lemma \ref{log}, (\ref{M1}), (\ref{M2}), (\ref{J1}) and (\ref{estint}) we find
\begin{align*}
J_1&=\sum\limits_{d_{i}|P(z)\atop{i=1,2,3,4}}
\frac{\lambda^-(d_{1})\lambda^+(d_{2})\lambda^+(d_{3})\lambda^+(d_{4})}  {\varphi(d_1)\varphi(d_2)\varphi(d_3)\varphi(d_4)}\int\limits_{|t|<\tau}\Theta(t)e(-N t)\\
&\times\left(\int\limits_{X/2}^{X}e(ty_1^c)dy_1\int\limits_{X/2}^{X}e(ty_2^c)dy_2
\int\limits_{X/2}^{X}e(ty_3^c)dy_3\int\limits_{X/2}^{X}e(ty_4^c)dy_4\right)dt\\
&=\sum\limits_{d_{i}|P(z)\atop{i=1,2,3,4}}
\frac{\lambda^-(d_{1})\lambda^+(d_{2})\lambda^+(d_{3})\lambda^+(d_{4})}  {\varphi(d_1)\varphi(d_2)\varphi(d_3)\varphi(d_4)}\left[\int\limits_{-\infty}^\infty\Theta(t)e(-Nt)\right.\\
&\times\left(\int\limits_{X/2}^{X}\int\limits_{X/2}^{X}\int\limits_{X/2}^{X}\int\limits_{X/2}^{X}
e(t(y_1^c+y_2^c+y_3^c+y_4^c))dy_1dy_2dy_3dy_4\right)dt\\
&+\mathcal{O}\left.\left(\vartheta X^{4-4c}\int\limits_{\tau}^\infty\frac{dt}{t^4}\right)\right]\\
&=\sum\limits_{d_{i}|P(z)\atop{i=1,2,3,4}}
\frac{\lambda^-(d_{1})\lambda^+(d_{2})\lambda^+(d_{3})\lambda^+(d_{4})}  {\varphi(d_1)\varphi(d_2)\varphi(d_3)\varphi(d_4)}\left(\int\limits_{X/2}^{X}\int\limits_{X/2}^{X}
\int\limits_{X/2}^{X}\int\limits_{X/2}^{X}\right.\\
&\times\int\limits_{-\infty}^\infty\Theta(t)e(t(y_1^c+y_2^c+y_3^c+y_4^c-N))dtdy_1dy_2dy_3dy_4
+\mathcal{O}\left(\vartheta X^{4-4c}\tau^{-3}\right)\Bigg)\\
\end{align*}
\begin{align*}
&=\sum\limits_{d_{i}|P(z)\atop{i=1,2,3,4}}
\frac{\lambda^-(d_{1})\lambda^+(d_{2})\lambda^+(d_{3})\lambda^+(d_{4})}
{\varphi(d_1)\varphi(d_2)\varphi(d_3)\varphi(d_4)}\\
&\times\left(\int\limits_{X/2}^{X}\int\limits_{X/2}^{X}\int\limits_{X/2}^{X}\int\limits_{X/2}^{X}
\theta(y_1^c+y_2^c+y_3^c+y_4^c-N)dy_1dy_2dy_3dy_4+\mathcal{O}(\vartheta X^{4-4c}\tau^{-3})\right)\\
&=\int\limits_{X/2}^{X}\int\limits_{X/2}^{X}\int\limits_{X/2}^{X}\int\limits_{X/2}^{X}
\theta(y_1^c+y_2^c+y_3^c+y_4^c-N)dy_1dy_2dy_3dy_4\\
&\times\sum\limits_{d_{i}|P(z)\atop{i=1,2,3,4}}
\frac{\lambda^-(d_{1})\lambda^+(d_{2})\lambda^+(d_{3})\lambda^+(d_{4})}
{\varphi(d_1)\varphi(d_2)\varphi(d_3)\varphi(d_4)}
+\mathcal{O}\left(\vartheta X^{4-4c}\tau^{-3}\log^4X\right)\,.
\end{align*}
The last formula, (\ref{tau}) and (\ref{Gama1,1J1est2})  imply

\begin{equation}\label{J1est}
\Gamma_1^{(1)}=B(X)\sum\limits_{d|P(z)}\frac{\lambda^-(d)}{\varphi(d)}
\left(\sum\limits_{d|P(z)}\frac{\lambda^+(d)}{\varphi(d)}\right)^3
+\mathcal{O}\left(\vartheta\frac{X^{4-c}}{(\log X)^{A-11}}\right)\,,
\end{equation}
where
\begin{equation}\label{B}
B(X)=\int\limits_{X/2}^{X}\int\limits_{X/2}^{X}\int\limits_{X/2}^{X}\int\limits_{X/2}^{X}
\theta(y_1^c+y_2^c+y_3^c+y_4^c-N)dy_1dy_2dy_3dy_4\,.
\end{equation}
According to (\cite{Zhai-Cao}, Lemma 8) we have
\begin{equation}\label{estB}
B(X)\gg\vartheta X^{4-c}\,.
\end{equation}
Let
\begin{equation}\label{Gpm}
G^\pm=\sum\limits_{d|P(z)}\frac{\lambda^{\pm}(d)} {\varphi(d)}\,.
\end{equation}
Thus from (\ref{J1est}) and (\ref{Gpm})  it follows
\begin{equation}\label{Ga1,1}
\Gamma_1^{(1)}=B(X)G^-\left(G^+\right)^3+\mathcal{O}\left(\vartheta\frac{X^{4-c}}{(\log X)^{A-11}}\right)\,.
\end{equation}
\section{Upper bound for $\mathbf{\Gamma_{1}^{(2)}}$.}
\indent

The treatment of the intermediate region depends on the following four lemmas.
\begin{lemma}\label{intL}
For the sums denoted by (\ref{L1}) and (\ref{L2}) we have
\begin{equation*}
\int\limits_0^1|L_j(t,X)|^2dt\ll X\log^5X\,,\;\; j=1,2.
\end{equation*}
\end{lemma}
\begin{proof} See (\cite{Tolev2} , Lemma 11).
\end{proof}

\begin{lemma}\label{summin}
Let $1<c<1603/1033$ . Then
\begin{equation*}
\sum\limits_{X/2<n_1, n_2, n_3, n_4\leq X}
\min\left(1, \frac{1}{|n_1^c+n_2^c-n_3^c-n_4^c|}\right)\ll X^{4-c}\log^5X\,.
\end{equation*}
\end{lemma}
\begin{proof} See (\cite{Zhai-Cao}, Theorem 1).
\end{proof}

\begin{lemma}\label{intL4}
For the sums denoted by (\ref{L1}) and (\ref{L2}) we have
\begin{equation*}
\int\limits_0^1|L_j(t,X)|^4dt\ll X^{4-c+\eta}\,,\;\; j=1,2,
\end{equation*}
where $\eta$ is defined by (\ref{eta}).
\end{lemma}
\begin{proof}
We only prove for $j=1$. The case for $j=2$  is analogous.

From (\ref{Rosser}), (\ref{L1}), Lemma \ref{tauest} and Lemma \ref{summin}
it follows
\begin{align*}
&\int\limits_{0}^{1}|L_1(t,X)|^4dt=\\
&=\sum\limits_{d_i|P(z)\atop{i=1,2,3,4}}\lambda^-(d_1)\cdots\lambda^-(d_4)
\sum\limits_{X/2<p_1,p_2,p_3,p_4\leq X\atop{p_i+2\equiv0\,(d_i), i=1,2,3,4}}\log p_1\cdots\log p_4
\int\limits_{0}^{1}e((p_1^c+p_2^c-p_3^c-p_4^c)t)dt\\
&\ll\sum\limits_{d_i\leq D\atop{i=1,2,3,4}}
\sum\limits_{X/2<p_1,p_2,p_3,p_4\leq X\atop{p_i+2\equiv0\,(d_i), i=1,2,3,4}}\log p_1\cdots\log p_4
\min\bigg(1,\frac{1}{|p_1^c+p_2^c-p_3^c-p_4^c|}\bigg)\\
&\ll(\log X)^4\sum\limits_{d_i\leq D\atop{i=1,2,3,4}}
\sum\limits_{X/2<n_1,n_2,n_3,n_4\leq X\atop{n_i+2\equiv0\,(d_i), i=1,2,3,4}}
\min\bigg(1,\frac{1}{|n_1^c+n_2^c-n_3^c-n_4^c|}\bigg)\\
&=(\log X)^4\sum\limits_{X/2<n_1,n_2,n_3,n_4\leq X}
\min\bigg(1,\frac{1}{|n_1^c+n_2^c-n_3^c-n_4^c|}\bigg)
\sum\limits_{d_1\leq D\atop{d_1|n_1+2}}1\cdots
\sum\limits_{d_4\leq D\atop{d_4|n_4+2}}1\\
&\ll(\log X)^4\sum\limits_{X/2<n_1,n_2,n_3,n_4\leq X}
\min\bigg(1,\frac{1}{|n_1^c+n_2^c-n_3^c-n_4^c|}\bigg)
\tau(n_1+2)\cdots\tau(n_4+2)
\end{align*}
\begin{align*}
&\ll X^{\eta/2}\sum\limits_{X/2<n_1,n_2,n_3,n_4\leq X}
\min\bigg(1,\frac{1}{|n_1^c+n_2^c-n_3^c-n_4^c|}\bigg)\\
&\ll X^{4-c+\eta}\,.
\end{align*}
\end{proof}

\begin{lemma}\label{LXest}
Assume that $\tau\leq|\alpha|\leq K$. Let $\beta(d)$ be complex number
defined for $d\leq D$, and let
\begin{equation}\label{beta}
\beta(d)\ll1.
\end{equation}
Then for the sum
\begin{equation}
L(\alpha,X)=\sum\limits_{d\leq D}\beta(d)\sum\limits_{X/2<p\leq X\atop{p+2\equiv0\,(d)}}e(p^c\alpha)\log p
\end{equation}
we have
\begin{equation*}
L(\alpha,X)\ll X^\eta\left(X^{1/3+c/2}DK^{1/2}+X^{3/4+c/6}D^{2/3}K^{1/6}+X^{1-c/6}D^{1/3}\tau^{-1/6}\right)\,,
\end{equation*}
where $\eta$ is defined by (\ref{eta}).
\end{lemma}
\begin{proof} See (\cite{Tolev2}, Lemma 15).
\end{proof}

We next treat $\Gamma_{1}^{(2)}$, defined by (\ref{Gamma1,2}).
We have
\begin{equation}\label{Gamma1,2est1}
\Gamma_{1}^{(2)}\ll\max\limits_{\tau\leq t\leq K}|L_1(t,X)|\int\limits_{\tau}^{K}|\Theta(t)||L_2(t,X)|^3dt\,.
\end{equation}
Using Cauchy's inequality we obtain
\begin{equation}\label{Cauchyinequality}
\int\limits_{\tau}^{K}|\Theta(t)||L_2(t,X)|^3dt
\ll\left(\int\limits_{\tau}^{K}|\Theta(t)||L_2(t,X)|^2dt\right)^{1/2}
\left(\int\limits_{\tau}^{K}|\Theta(t)||L_2(t,X)|^4dt\right)^{1/2}\,.
\end{equation}
On the one hand from (\ref{eps}), (\ref{K}), Lemma \ref{Fourier} and Lemma \ref{intL} it follows
\begin{align}\label{midint1}
\int\limits_{\tau}^{K}|\Theta(t)||L_2(t,X)|^2dt
&\ll\vartheta\int\limits_{\tau}^{1/\vartheta}|L_2(t,X)|^2dt
+\int\limits_{1/\vartheta}^{K}|L_2(t,X)|^2\frac{dt}{t}\nonumber\\
&\ll\vartheta\sum\limits_{0\leq n\leq1/\vartheta}\int\limits_{n}^{n+1}|L_2(t,X)|^2dt
+\sum\limits_{1/\vartheta-1\leq n\leq K}\frac{1}{n}\int\limits_{n}^{n+1}|L_2(t,X)|^2dt\nonumber\\
&\ll X\log^6X\,.
\end{align}
On the other hand (\ref{eps}), (\ref{K}), Lemma \ref{Fourier} and Lemma \ref{intL4} give us
\begin{align}\label{midint2}
\int\limits_{\tau}^{K}|\Theta(t)||L_2(t,X)|^4dt
&\ll\vartheta\int\limits_{\tau}^{1/\vartheta}|L_2(t,X)|^4dt
+\int\limits_{1/\vartheta}^{K}|L_2(t,X)|^4\frac{dt}{t}\nonumber\\
&\ll\vartheta\sum\limits_{0\leq n\leq1/\vartheta}\int\limits_{n}^{n+1}|L_2(t,X)|^4dt
+\sum\limits_{1/\vartheta-1\leq n\leq K}\frac{1}{n}\int\limits_{n}^{n+1}|L_2(t,X)|^4dt\nonumber\\
&\ll X^{4-c+\eta}\log X\,,
\end{align}
where $\eta$ is defined by (\ref{eta}).\\
Therefore by  (\ref{tau}) -- (\ref{eta}), (\ref{Gamma1,2est1}) -- (\ref{midint2}) and by Lemma \ref{LXest} we obtain
\begin{equation}\label{Gamma1,2est2}
\Gamma_1^{(2)}\ll\vartheta\frac{X^{4-c}}{\log^5X}\,.
\end{equation}
Summarizing (\ref{Gama1}), (\ref{Gamma1,3est}), (\ref{Ga1,1}) and (\ref{Gamma1,2est2}) we find
\begin{equation}\label{Gama1est}
\Gamma_1=B(X)G^-(G^+)^3+\mathcal{O}\bigg( \vartheta \frac{ X^{4-c}}{\log^5X}\bigg)\,.
\end{equation}
\section{Proof of the Theorem.}
\indent

Since $\Gamma_1=\Gamma_2=\Gamma_3=\Gamma_4$   and $\Gamma_5$  is estimated in the same way then from (\ref{Gamma}), (\ref{Gammaest1}), (\ref{Gammawaveest}), (\ref{Gamma0}) and (\ref{Gama1est}) we obtain
\begin{equation}\label{Gammaest2}
\Gamma\geq B(X)W+\mathcal{O}\bigg( \vartheta\frac{X^{4-c}}{\log^5X}\bigg)\,,
\end{equation}
where
\begin{equation}\label{WGamma}
W=4\left(G^+\right)^3\bigg(G^--\frac{3}{4}G^+\bigg )\,.
\end{equation}
We put
\begin{equation}\label{Fzs}
\mathcal{F}(z)=\prod\limits_{2<p\le z}\bigg(1-\frac{1}{p-1}\bigg)\,,\quad
s=\frac{\log D}{\log z}\,.
\end{equation}
Let $f(s)$ and $F(s)$ are the lower and the upper functions of the linear sieve.
Using (\ref{Gpm}) and (\cite{Bru}, Lemma 10) we obtain
\begin{align}\label{G_iF_i}
  \mathcal{F}(z)\bigg (f(s)+&\mathcal{O}\big((\log X)^{-1/3}\big)\bigg )\nonumber\\
  &\quad\quad\quad\le G^{-} \le \mathcal{F}(z)
  \le G^{+}\nonumber\\
  &\quad\quad\quad\quad\quad\quad\quad\le \mathcal{F}(z)\bigg (F(s)+\mathcal{O}\left((\log X)^{-1/3}\right)\bigg )\,.
\end{align}
To estimate $W$ from below we shall use the inequalities (see (\ref{G_iF_i})):
\begin{align}
\begin{split}\label{G1theta1}
   G^--\frac{2}{3}G^+&\geq \mathcal{F}(z)\bigg (f(s)-\frac{3}{4}F(s)+\mathcal{O}\big((\log X)^{-1/3}\big)\bigg )\\
   G^+&\geq \mathcal{F}(z)\,.
\end{split}
\end{align}
Then from (\ref{WGamma}) and (\ref{G1theta1}) it follows
\begin{equation}\label{WGammaest}
W\geq4\mathcal{F}^4(z)\bigg (f(s)-\frac{3}{4}F(s)+\mathcal{O}\big((\log X)^{-1/3}\big)\bigg)\,.
\end{equation}
Hence, using (\ref{Gammaest2}) and (\ref{WGammaest}) we get
\begin{equation}\label{Gammaest3}
\Gamma\geq4B\mathcal{F}^4(z)\bigg(f(s)-\frac{2}{3}F(s)+\mathcal{O}\big((\log X)^{-1/3}\big)\bigg)+\mathcal{O}\bigg( \vartheta\frac{X^{4-c}}{\log^5X}\bigg)\,.
\end{equation}
For $2\leq s\leq3$ we have
\begin{equation*}
f(s)=\frac{2e^\gamma\log(s-1)}{s}\,, \qquad F(s)=\frac{2e^\gamma}{s}
\end{equation*}
($\gamma$ denotes Euler's constant).
We choose
\begin{equation*}
s=2.95.
\end{equation*}
Then by (\ref{z}), (\ref{D}) and (\ref{Fzs}) we find
\begin{equation*}
\beta=0.030477\,.
\end{equation*}
It is not difficult to compute that for sufficiently large $X$ we have
\begin{equation}\label{ineq}
f(s)-\frac{2}{3}F(s)>10^{-5}.
\end{equation}
It remains to notice that
\begin{equation}\label{Fzest}
\mathcal{F}(z)\asymp\frac{1}{\log X}\,.
\end{equation}
Therefore, using (\ref{z}), (\ref{estB}), (\ref{Gammaest3}) -- (\ref{Fzest}) we obtain
\begin{equation}\label{final}
\Gamma\gg\vartheta\frac{X^{4-c}}{\log^4X}.
\end{equation}
From (\ref{eps}) and (\ref{final}) it follows that $\Gamma \rightarrow\infty$ as $X\rightarrow\infty$.

Bearing in mind (\ref{eps}), (\ref{Gamma}) and (\ref{final}) we conclude that for some
constant $c_0>0$ there are at least $c_0X^{4-c}\log^{-A-9}X$ triples of primes $p_1, p_2, p_3$
satisfying $X/2< p_1, p_2, p_3,p_4\le X,\;|p_1^c+p_2^c+p_3^c+p_4^c-N|<\vartheta$
and such that for every prime
factor $p$ of  $p_j + 2,\,j=1,2,3,4$ we have $ p\geq X^{0.030477}$.

The proof of the Theorem is complete.

\vskip20pt
\footnotesize
\begin{flushleft}
S. I. Dimitrov\\
Faculty of Applied Mathematics and Informatics\\
Technical University of Sofia \\
8, St.Kliment Ohridski Blvd. \\
1756 Sofia, BULGARIA\\
e-mail: sdimitrov@tu-sofia.bg\\
\end{flushleft}

\begin{thebibliography}{}
\bibitem{Baker-Weingartner} R. Baker, A. Weingartner,
{\it A ternary diophantine inequality over primes},
Acta Arith., {\bf162}, (2014), 159 -- 196.

\bibitem{Bru} J. Br\"{u}dern, E. Fouvry, {\it Lagrange's Four Squares Theorem with almost prime
variables}, J. Reine Angew. Math., {\bf 454}, (1994), 59 -- 96.

\bibitem{Cai} Y. Cai, {\it On a diophantine inequality involving prime numbers} (in Chinese),
Acta Math Sinica, {\bf39}, (1996), 733 -- 742.

\bibitem{Dimitrov}S. I. Dimitrov, {\it Studying diophantine inequalities and arithmetical progressions using number theory methods},
 Thesis, Technical University - Sofia, (2016), (in Bulgarian).

\bibitem{Iwa1} H. Iwaniec, {\it Rosser's sieve}, Acta Arith., {\bf 36}, (1980), 171 -- 202.

\bibitem{Iwa2} H. Iwaniec, {\it A new form of the error
term in the linear sieve}, Acta Arith., {\bf 37}, (1980), 307 -- 320.

\bibitem{Ku-Ne} A. Kumchev, T. Nedeva,
{\it On an equation with prime numbers}, Acta Arith., {\bf 83},
(1998), 117 -- 126.

\bibitem{Kumchev} A. Kumchev, {\it A diophantine inequality involving prime powers},
Acta Arith., {\bf 89}, (1999), 311 -- 330.

\bibitem{Mu} Q. Mu, {\it On a diophantine inequality over primes},
Adv. Math. (China), {\bf 44}(4), (2015), 621 -- 637.

\bibitem{Shapiro} I. I. Piatetski-Shapiro, {\it On a variant of the Waring-Goldbach problem},
Mat. Sb., {\bf30}, (1952), 105 -- 120, (in Russian).

\bibitem{Segal} B. Segal, {\it On a theorem analogous to Waring's theorem},
Dokl. Akad. Nauk SSSR (N. S.), {\bf2}, (1933), 47 -- 49, (in Russian).

\bibitem{Titchmarsh}E. Titchmarsh, {\it The Theory of the Riemann Zeta-function}
(revised by D. R. Heath-Brown), Clarendon Press, Oxford (1986).

\bibitem{Tolev1} D. Tolev, {\it On a diophantine inequality
involving prime numbers}, Acta Arith., {\bf 61}, (1992), 289 -- 306.

\bibitem{Tolev2} D. Tolev, {\it On a diophantine inequality with prime
numbers of a special type}, arXiv:1701.07652v1 [math.NT].

\bibitem{Zhai-Cao} W. Zhai, X. Cao, {\it On a diophantine inequality over primes},
Adv. Math. (China), {\bf 32}(1), (2003), 63 -- 73.

\end{thebibliography}
\end{document}